\documentclass[letterpaper]{article}
\usepackage{aaai18}
\usepackage{times}
\usepackage{helvet}
\usepackage{courier}
\usepackage{url}
\usepackage{graphicx}
\usepackage{amssymb,amsthm}
\usepackage{mathtools}
\usepackage[T1]{fontenc}
\usepackage{microtype}
\title{A New Form of Williamson's Product Theorem}
\author{Curtis Bright\\University of Waterloo}

\newcommand{\Z}{\mathbb{Z}}

\newtheorem{theorem}{Theorem}
\newtheorem{corollary}{Corollary}

\nocopyright

\begin{document}
\maketitle
\begin{abstract}
A form of Williamson's product theorem which applies to
Williamson matrices of even order is presented.
\end{abstract}
Four symmetric and circulant $n\times n$ matrices with $\pm1$ entries
are known as \emph{Williamson matrices} if they satisfy
\[ A^2 + B^2 + C^2 + D^2 = 4nI_n \]
where $I_n$ is the $n\times n$ identity matrix.  Such matrices were
first introduced by Williamson, who showed that they can be used
to constuct a Hadamard matrix of order~$4n$
and derived properties which such matrices must satisfy~\cite{williamson1944hadamard}.
Since Williamson matrices are circulant they are defined in terms of their
first row, e.g.,
$A=\operatorname{circ}(a_0,\dotsc,a_{n-1})$.
Because of this it is convenient to instead think of $A$, $B$, $C$, $D$ as four
sequences of length $n$ and refer to them as
\emph{Williamson sequences}~\cite{brightthesis}.

Using this terminology, Williamson proved the following result about
the entries of Williamson sequences which we call \emph{Williamson's product theorem}.

\begin{theorem}\label{thm:willprododd}
If\/ $A$, $B$, $C$, $D$ is a Williamson sequence of odd order\/ $n$ then\/
$a_ib_ic_id_i=-a_0b_0c_0d_0$ for\/ $1\leq i<n/2$.
\end{theorem}

We now prove a version of this theorem for even $n$:

\begin{theorem}\label{thm:willprodeven}
If\/ $A$, $B$, $C$, $D$ is a Williamson sequence of even order\/ $n=2m$ then\/
$a_ib_ic_id_i=a_{i+m}b_{i+m}c_{i+m}d_{i+m}$ for\/ $0\leq i<m$.
\end{theorem}

\begin{proof}
We can equivalently consider members of Williamson sequences to be
elements of the group ring $\Z[C_n]$ where $C_n$ is a cyclic group
of order $n$ with generator $u$.  In such a formulation we have
$X=x_0+x_1u+\dotsb+x_{n-1}u^{n-1}$ and Williamson sequences are
quadruples $(A,B,C,D)$ whose members have $\pm1$ coefficients,
whose coefficients form symmetric sequences of length $n$, and which satisfy
\[ A^2 + B^2 + C^2 + D^2 = 4n . \]
Let $P_X=\sum_{x_i=1} u^i$ (with the sum over $0\leq i<n$) and let $p_X$ denote the number of positive
coefficients in $X$.
As shown in~\cite[14.2.20]{hall1998combinatorial} we have that
$P_A^2 + P_B^2 + P_C^2 + P_D^2$ is equal to
\[ (p_A+p_B+p_C+p_D-n)\sum_{i=0}^{n-1}u^i + n . \tag{1}\label{eq:prodeq1} \]
Furthermore, by the fact that $P_X^2\equiv\sum_{x_i=1} u^{2i}\pmod{2}$,
$P_A^2 + P_B^2 + P_C^2 + P_D^2$ is
congruent to
\[ \sum_{a_i=1} u^{2i} + \sum_{b_i=1} u^{2i} + \sum_{c_i=1} u^{2i} + \sum_{d_i=1} u^{2i} \pmod{2} . \tag{2}\label{eq:prodeq2} \]
Now, if $n$ is even then~\eqref{eq:prodeq1} reduces to
\[ (p_A+p_B+p_C+p_D)\sum_{i=0}^{n-1}u^i \pmod{2} \]
so all coefficients are the same mod $2$.  Since by~\eqref{eq:prodeq2} the
coefficients with odd index are $0$ mod $2$, all coefficients in~\eqref{eq:prodeq1}
and~\eqref{eq:prodeq2} must be $0$ mod $2$.

Note that $u^k=u^{2i}$ has exactly 2 solutions for given even~$k$ with $0\leq k<n$, namely,
$i=k/2$ and $i=(k+n)/2$.  Then~\eqref{eq:prodeq2} can be rewritten as
\[ \sum_{a_{k/2}=1} u^k + \sum_{a_{(k+n)/2}=1} u^k + \dotsb + \sum_{d_{(k+n)/2}=1} u^k \pmod{2} \]
where the sums are over the even $k$ with $0\leq k<n$.
Since each coefficient must be $0$ mod $2$, there must be an even number of $1$s
among the entries $a_{k/2}$, $a_{(k+n)/2}$, $\dotsc$, $d_{(k+n)/2}$ for each even $k$ with $0\leq k<n$, i.e.,
\[ a_{k/2}a_{(k+n)/2}b_{k/2}b_{(k+n)/2}c_{k/2}c_{(k+n)/2}d_{k/2}d_{(k+n)/2} = 1 . \]
The required result is a rearrangement of this and rewriting with
the definition $i=k/2$.
\end{proof}

Following~\cite[Def.~3]{dokovic2015compression}, the $2$\nobreakdash-compression
of the sequence $A=[a_0,\dotsc,a_{n-1}]$ of even length
$n=2m$ is the sequence $A'$ of length $m$ whose $i$th entry is $a'_i\coloneqq a_i+a_{i+m}$ for $0\leq i<m$.
This definition allows us to state Theorem~\ref{thm:willprodeven} in an alternative useful form.

\begin{corollary}\label{cor:willprod}
If\/ $A'$, $B'$, $C'$, $D'$ is a\/ $2$\nobreakdash-compression of a Williamson sequence of even order\/ $n=2m$
then\/ $A'+B'+C'+D'\equiv[0,\dotsc,0]\pmod{4}$.
\end{corollary}

\begin{proof}
Let $N_+$ and $N_-$ denote the number of $1$s and $-1$s in the eight Williamson sequence entries
$a_i$, $b_i$, $c_i$, $d_i$, $a_{i+m}$, $b_{i+m}$, $c_{i+m}$, and $d_{i+m}$,
where $0\leq i<m$.  We have that $N_++N_-=8$ and that
$N_+-N_-=a'_i+b'_i+c'_i+d'_i$
(the sum of the above eight Williamson sequence entries).
Thus the $i$th entry of $A'+B'+C'+D'$ is
$N_+-N_-=N_+-(8-N_+)=2N_+-8\equiv0\pmod 4$ since
Theorem~\ref{thm:willprodeven} implies that $N_+$ must be even.
\end{proof}

\bibliography{williamson-product-theorem}
\bibliographystyle{aaai}

\end{document}